\documentclass[english]{article}
\usepackage[T1]{fontenc}
\usepackage[latin9]{inputenc}
\usepackage{xcolor}
\usepackage{pdfcolmk}
\usepackage{array}
\usepackage{multirow}
\usepackage{amsmath}
\usepackage{amsthm}
\usepackage{amssymb}
\PassOptionsToPackage{normalem}{ulem}
\usepackage{ulem}
\usepackage{authblk}

\makeatletter

\providecommand{\tabularnewline}{\\}
\providecolor{lyxadded}{rgb}{0,0,1}
\providecolor{lyxdeleted}{rgb}{1,0,0}

\DeclareRobustCommand{\lyxsout}[1]{\ifx\\#1\else\sout{#1}\fi}

\theoremstyle{plain}
\newtheorem{thm}{\protect\theoremname}
\theoremstyle{plain}
\newtheorem{cor}{\protect\corollaryname}

\makeatother

\usepackage{babel}
\providecommand{\corollaryname}{Corollary}
\providecommand{\theoremname}{Theorem}

\begin{document}
\title{D-optimal designs for the Mitscherlich non-linear regression function}

\author[1]{\small Maliheh Heidari}
\author[1]{\small Md Abu Manju}
\author[2]{\small Pieta C. IJzerman-Boon}
\author[1]{\small Edwin R. van den Heuvel}

\affil[1]{\footnotesize Department of Mathematics and Computer Science, Eindhoven University of Technology, 5600 MB Eindhoven, The Netherlands}
\affil[2]{\footnotesize Center for Mathematical Sciences, MSD, 5342 CC Oss, The Netherlands}

\maketitle
\begin{abstract}
Mitscherlich's function is a well-known three-parameter non-linear
regression function that quantifies the relation between a stimulus
or a time variable and a response. Optimal designs for this function
have been constructed only for normally distributed responses with
homoscedastic variances. In this paper, we construct D-optimal designs
for discrete and continuous responses having their distribution function
in the exponential family. We also demonstrate the connection with
D-optimality for weighted linear regression.
\end{abstract}

\section{Introduction\label{sec:intro}}

In different fields of science (e.g., chemistry, biology, medicine,
and pharmacology) the relation between a stimulus or a time variable
($x$) and a response variable $(y$) is being studied. For instance,
the three-parameter Michaelis-Menten curve $\mathbb{E}(y|x)=\beta_{1}+\beta_{2}x/[\beta_{3}+x${]}
is frequently used for chemical and biological applications \cite{key-18,key-8,key-30},
the four-parameter logistic growth curve $\mathbb{E}(y|x)=\beta_{1}+(\beta_{4}-\beta_{1})[1+(x/\beta_{2})^{\beta_{3}}]^{-1}$
is typically used for biological assays \cite{key-3,key-10}, the
four-parameter non-linear exponential growth or decay curve $\mathbb{E}(y|x)=\beta_{1}+\beta_{2}x+\beta_{3}\exp\{\beta_{4}x\}$
is used in biology and medical sciences \cite{key-4,key-5}, and the
three-parameter one-compartmental model $\mathbb{E}(y|x)=\beta_{1}[\exp\{-\beta_{2}x\}-\exp\{-\beta_{3}x\}]/[\beta_{3}-\beta_{2}]$
is often used in pharmacokinetics \cite{key-9,key-2,key-11}. These
relations are all non-linear in both the stimulus or time variable
and the model parameters.

Precise estimation of non-linear models may require a substantial
amount of testing. Designing optimal experiments may therefore help
reduce testing and possibly reduce also other resources (e.g., time,
costs). A parameter estimation criterion for optimal designs is D-optimality
\cite{key-13}, which maximizes the determinant of $\boldsymbol{X}^{T}\boldsymbol{X}$
for linear regression functions, with $\boldsymbol{X}$ the design
matrix. For non-linear functions D-optimality is obtained by maximizing
the determinant of the Fisher information matrix \cite{key-13}. D-optimal
designs have been studied for different types of non-linear functions
for both continuous and count responses.

Under assumption of normality, $y_{i}=\mathbb{E}(y_{i}|x_{i})+\varepsilon_{i}$,
with $\varepsilon_{i}\sim\mathcal{N}(0,\sigma^{2})$ i.i.d., \cite{key-8}
provided D-optimal designs for the two-parameter ($\beta_{1}=0$)
Michaelis-Menten curve, while \cite{key-30} discussed D-optimal designs
for this two-parameter Michaelis-Menten curve under heteroscedastic
residual errors, i.e., $\varepsilon_{i}\sim\mathcal{N}(0,\nu(\mathbb{E}(y_{i}|x_{i})))$,
with $\nu$ a known function. In \cite{key-7}, a D-optimal design
for the three ($\beta_{1}=0$) and the four parameter logistic growth
curve was provided, respectively. In \cite{key-6}, this work on logistic
curves was extended to heteroscedastic residuals, i.e., $\varepsilon_{i}\sim\mathcal{N}(0,\sigma^{2}\mathbb{E}(y_{i}|x_{i})[1-\mathbb{E}(y_{i}|x_{i})])$,
when $\beta_{1}=0$ and $\beta_{4}=1$ holds. Under the same parameter
restrictions, \cite{key-6} also provided D-optimal designs for the
asymmetric logistic growth curve, i.e., $\mathbb{E}(y_{i}|x_{i})=[1+(x/\beta_{2})^{\beta_{3}}]^{-r}$,
with $r>0$. In \cite{key-10}, D-optimal designs for the full four
and five parameter logistic growth curve $\mathbb{E}(y_{i}|x_{i})=\beta_{1}+(\beta_{4}-\beta_{1})[1+(x/\beta_{2})^{\beta_{3}}]^{-r}$
were studied with residuals having a heteroscedastic variance of the
form $[\mathbb{E}(y_{i}|x_{i})]^{\gamma}$, with $\gamma>0$. D-optimal
designs for the two-parameter (where $\beta_{2}=0$ and $\beta_{3}=1$
and where $\beta_{1}=\beta_{2}=0$) and three-parameter ($\beta_{2}=0$)
exponential decay model were provided by \cite{key-5} under the assumption
of homoscedastic residuals, while \cite{key-4} provided a D-optimal
design for the full four-parameter exponential growth model (among
others), also under homoscedastic residuals. Finally, \cite{key-21}
discussed D-optimal designs for the three-parameter one-compartmental
model under homoscedastic residuals, while \cite{key-10} studied
D-optimality of this compartmental model under heteroscedastic residuals
using $[\mathbb{E}(y_{i}|x_{i})]^{\gamma}$ (again).

For count responses $y_{i}$, the Poisson, Binomial, and Negative
Binomial distributions have been used frequently \cite{key-12,key-15,key-16,key-17,key-19},
but these papers discuss optimal designs for forms of $\mathbb{E}(y_{i}|x_{i})$
that can be rewritten into a linear function in the parameters, i.e.,
satisfying the definition of generalized linear models \cite{key-14}.
Interestingly though, \cite{key-19} provided D-optimal designs for
the class of generalized linear models with distributions in the exponential
family a few years earlier. Contrary to the work on generalized linear
models, \cite{key-18} discussed D-optimal designs for mixed effects
Poisson regression with the full three-parameter Michaelis-Menten
curve. The random part only affected the constant or intercept $\beta_{1}$
and they also discussed designs without this random component.

One specific or special non-linear regression function is the three-parameter
Mitscherlich function \cite{key-1}, given by $\mathbb{E}(y|x)=\beta_{1}+\beta_{2}\exp\{\beta_{3}x\}$,
with $\beta_{1}\in\mathbb{R}$, $\beta_{2}\neq0,\beta_{3}\neq0$,
and with $x$ the stimulus or the logarithmically transformed stimulus
variable. Note that the original formulation of the Mitscherlich function
in \cite{key-31} assumed that the parameters $\beta_{2}$ and $\beta_{3}$
were both negative. In some areas \cite{key-5,key-41}, the Mitscherlich
function is referred to as the three parameter decay model when the
variable $x$ is time. In that case the parameter $\beta_{3}$ is
typically considered negative. The reason that the Mitscherlich function
is special, is that it can be naturally used to investigate violations
of \textit{linearity} of a measurement system in different directions,
which is less obvious for the other non-linear functions just discussed.
Indeed, linearity can be obtained in two ways:
\begin{equation}
\begin{array}{rl}
\beta_{3}=1: & \mathbb{E}(y|\log(x))=\beta_{1}+\beta_{2}x,\\
\beta_{1}=0: & \log(\mathbb{E}(y|x))=\log(\beta_{2})+\beta_{3}x,
\end{array}\label{eq:linearity}
\end{equation}
with the $\log$ the natural logarithm. In case both constraints $\beta_{1}=0$
and $\beta_{3}=1$ are satisfied, the system may be referred to as
\textit{proportional} to stimulus $x$.

As far as we know, D-optimal designs for the Mitscherlich non-linear
function have only be discussed under the assumption of a normally
distributed response $y$ with homoscedastic residual variances \cite{key-1,key-5,key-41}.
Here we will generalize these D-optimal designs for estimation of
the Mitscherlich function, when the discrete or continuous distribution
function for the response $y$ is from the exponential family in its
natural form \cite{key-14,key-33}. We also consider the situation
where the dispersion parameter is not known.

The next section will introduce our generalized non-linear model,
the log-likelihood function, Fisher's information matrix, and the
D-optimality criterion. In Section 3 we will construct the D-optimal
design for the Mitscherlich non-linear function using three stimuli
levels (minimally D-optimal \cite{key-21}). We also provide examples
for the well-known distributions in the exponential family of distributions.
In Section 4 we discuss transformations of the Mitscherlich non-linear
function and how our work is connected to earlier results in literature,
showing that the D-optimal designs in \cite{key-1,key-5,key-41} are
obtained with our results. Furthermore, we will show that our D-optimal
design can be constructed from a D-optimal design for weighted linear
regression, extending \cite{key-1} directly to the distributions
in the exponential family. However, when heteroscedastic residual
variances are introduced, D-optimality can not be obtained through
weighted linear regression anymore. We finalize with Section 5 summarizing
and discussing our work.

\section{Statistical model}

Let $y_{ij}$ be response $j\in\{1,2,...,n_{i}\}$ at stimulus $x_{i}$,
$i\in\{1,2,...,m\}$, all being mutually independently distributed.
The distribution of $y_{ij}$ is an element of the exponential family
having density $f(y|\theta_{i},\phi)=\exp\{[y\theta_{i}-b(\theta_{i})]/a(\phi)+c(y,\phi)\}$,
with $y\in\mathbb{R}$, $\theta_{i}$ an unknown parameter that will
depend on stimulus $x_{i}$, $\phi$ an (un)known dispersion parameter,
and $a(\cdot)$, $b(\cdot)$ and $c(\cdot,\cdot)$ known functions
\cite{key-14}. It is assumed that the range of $y$ does not depend
on $\theta_{i}$ and $\phi$. Furthermore, function $b(\cdot)$ is
at least twice differentiable, with $b'(\cdot)$ and $b''(\cdot)$
the first and second derivative. As a consequence, we have $\mathbb{E}(y_{ij}|x_{i})\equiv\mu_{i}=b'(\theta_{i})$
and $\mathsf{VAR}(y_{ij}|x_{i})=b''(\theta_{i})a(\phi)$. Using the
canonical link function $g$, the relation between $\theta_{i}$ and
$\mu_{i}$ is given by $\theta_{i}=g(\mu_{i})$. Our model includes
the well-known distributions Poisson, Binomial, Negative Binomial,
Gaussian, Gamma, and Inverse Gaussian with their canonical link functions.
However, there exists a more general formulation of the exponential
family of distributions of the form $f(y|\theta_{i})=\exp\{T(y)\eta(\theta_{i})-A(\theta_{i})+B(y)\}$,
but we have selected its more restrictive natural form with $\theta_{i}$
the canonical parameter (when $\phi$ is known \cite{key-33}). Furthermore,
if $\phi$ is unknown our formulation may not be a two-parameter exponential
family anymore \cite{key-14}. Irrespective of its formal definition,
we will focus on densities $f(y|\theta_{i},\phi)=\exp\{[y\theta_{i}-b(\theta_{i})]/a(\phi)+c(y,\phi)\}$
where $\phi$ is allowed to be unknown.

The Mitscherlich function we will study is $\mu_{i}=\beta_{1}+\beta_{2}x_{i}^{\beta_{3}}$
with constraints $\beta_{2}>0$, $\beta_{3}>0$, and $x_{i}\geq0$
the stimulus of interest. Note that we allow a stimulus that can be
equal to zero, which was not implemented in earlier formulations.
Restrictions on parameter $\beta_{1}$ are determined by the type
of distribution for $y_{ij}$. For instance, $\beta_{1}\in\mathbb{R}$
is allowed for the normal distribution, $\beta_{1}\geq0$ is needed
for the Poisson distribution, and $\beta_{1}>0$ is required for the
Gamma distribution. Our choice for the Mitscherlich function fits
very well with measurement system analysis where we expect typically
non-negative values when we choose certain levels for the stimulus.
Thus we will assume that $\beta_{1}\geq0$.

\subsection{Maximum likelihood estimation}

If we define $\boldsymbol{y}_{i}=(y_{i1},y_{i2},...,y_{in_{i}})^{T}$,
$\boldsymbol{y}=(\boldsymbol{y}_{1},\boldsymbol{y}_{2},...,\boldsymbol{y}_{m})^{T}$,
and $\boldsymbol{\beta}=(\beta_{1},\beta_{2},\beta_{3})^{T}$, the
log-likelihood function can be written as
\begin{equation}
\begin{array}{rcl}
\ell\left(\boldsymbol{\beta},\phi|\boldsymbol{y}\right) & = & \sum\limits _{i=1}^{m}\sum\limits _{j=1}^{n_{i}}\left[(y_{ij}\theta_{i}-b(\theta_{i}))/a(\phi)+c(y_{ij},\phi)\right]\\
 & = & \tfrac{1}{a(\phi)}\sum\limits _{i=1}^{m}\left[y_{i.}g\left(\mu_{i}\right)-n_{i}b(g(\mu_{i}))\right]+\sum\limits _{i=1}^{m}\sum\limits _{j=1}^{n_{i}}c(y_{ij},\phi),
\end{array}\label{eq:likelihood}
\end{equation}
where $y_{i.}=\sum_{j=1}^{n_{i}}y_{ij}$ is the sum of the observations
at stimulus $x_{i}$. The maximum likelihood estimates (MLEs) for
the parameters $\boldsymbol{\beta}$ and $\phi$ can be obtained by
solving the following likelihood equations:
\begin{equation}
\begin{array}{l}
\ell_{\beta_{k}}^{\prime}=\frac{\partial\ell\left(\boldsymbol{\beta},\phi|\boldsymbol{y}\right)}{\partial\beta_{k}}=\tfrac{1}{a(\phi)}\sum\limits _{i=1}^{m}\left(y_{i.}-n_{i}\mu_{i}\right)g^{\prime}\left(\mu_{i}\right)\frac{\partial\mu_{i}}{\partial\beta_{k}}=0\,\,\,\forall k=1,2,3,\\
\ell_{\phi}^{\prime}=\frac{\partial\ell\left(\boldsymbol{\beta},\phi|\boldsymbol{y}\right)}{\partial\phi}=-\frac{a'\left(\phi\right)}{a^{2}\left(\phi\right)}\sum\limits _{i=1}^{m}\left[y_{i.}g(\mu_{i})-n_{i}b(g(\mu_{i}))\right]+\sum\limits _{i=1}^{m}\sum\limits _{j=1}^{n_{i}}c'(y_{ij},\phi)=0,
\end{array}\label{eq:score-functions}
\end{equation}
where $g'(\mu_{i})=\partial g(\mu_{i})/\partial\mu_{i}$, $a'(\phi)=\partial a(\phi)/\partial\phi$,
and $c^{\prime}(y_{ij},\phi)=\partial c(y_{ij},\phi)/\partial\phi$.
The $4\times4$ Fisher information matrix $I_{4\times4}(\boldsymbol{\beta},\phi)$
is obtained by the (negative) expected values of the derivatives of
the score functions in (\ref{eq:score-functions}), but they are also
equal to the variances and covariances of the score functions, Theorem
1.1, page 406 in \cite{key-32}. Using the derivatives of the score
functions and taking expectations (see Appendix A), the variances
and covariances of the score functions become
\begin{equation}
\begin{array}{l}
\mathrm{\mathsf{VAR}}(\ell_{\beta_{k}}^{\prime})=\tfrac{1}{a(\phi)}\sum\limits _{i=1}^{m}n_{i}g'(\mu_{i})\left(\frac{\partial\mu_{i}}{\partial\beta_{k}}\right)^{2},\\
\mathsf{VAR}(\ell_{\phi}^{\prime})=\sum\limits _{i=1}^{m}\sum\limits _{j=1}^{n_{i}}\left[\left(\frac{a''\left(\phi\right)}{a'\left(\phi\right)}-\frac{2a'\left(\phi\right)}{a\left(\phi\right)}\right)\mathrm{\mathbb{E}}\left(\frac{\partial c(y_{ij},\phi)}{\partial\phi}\right)-\mathbb{E}\left(\frac{\partial^{2}c(y_{ij},\phi)}{(\partial\phi)^{2}}\right)\right],\\
\mathsf{COV}(\ell_{\beta_{r}}^{\prime},\ell_{\beta_{s}}^{\prime})=\tfrac{1}{a(\phi)}\sum\limits _{i=1}^{m}n_{i}g'(\mu_{i})\left(\frac{\partial\mu_{i}}{\partial\beta_{r}}\right)\left(\frac{\partial\mu_{i}}{\partial\beta_{s}}\right),\quad r\neq s,\\
\mathsf{COV}(\ell_{\beta_{k}}^{\prime},\ell_{\phi}^{\prime})=0.
\end{array}\label{eq:Fisher-Information}
\end{equation}
Orthogonality of the score functions for the location parameters $\boldsymbol{\beta}$
and the score function for the dispersion parameter $\phi$ has been
obtained earlier \cite{key-22}.

\textbf{Remark:} The Fisher information matrix $I_{4\times4}(\boldsymbol{\beta},\phi)$
reduces to a $3\times3$ matrix $I_{3\times3}(\boldsymbol{\beta})/a(\phi)$, when $\phi$ would be
known (e.g., $\phi=1$). This matrix, $I_{3\times3}(\boldsymbol{\beta})$, is independent of $\phi$ and will be fully determined by the score functions $\ell_{\beta_{k}}^{\prime}$
in (\ref{eq:Fisher-Information}). Note that $\mathsf{COV}(\ell_{\beta_{k}}^{\prime},\ell_{\phi}^{\prime})=0$
for all $k\in\{1,2,3\}$, implies that the covariance of the MLEs
for $\beta_{k}$ and $\phi$ is zero too, but this does not necessarily
imply that the variance $\mathsf{VAR}(\hat{\boldsymbol{\beta}})$
of MLE $\hat{\boldsymbol{\beta}}$ is independent of $\phi$ or the
variance $\mathsf{VAR}(\hat{\phi})$ of MLE $\hat{\phi}$ is independent
of $\boldsymbol{\beta}$, since the corresponding elements of the
inverse Fisher information may still depend on $\phi$ or $\boldsymbol{\beta}$
through its density, respectively.

\subsection{D-optimality criterion}

D-optimality is defined by maximizing the determinant of the Fisher
information matrix $I_{4\times4}(\boldsymbol{\beta},\phi)$, see \cite{key-13,key-51}.
Due to the (asymptotic) independence of the ML estimators $\hat{\boldsymbol{\beta}}$
and $\hat{\phi}$, the determinant of the Fisher information matrix
can be rewritten as $|I_{4\times4}(\boldsymbol{\beta},\phi)|=\mathsf{VAR}(\ell_{\phi}^{\prime})|I_{3\times3}(\boldsymbol{\beta})|/a(\phi)$.
In case the variance $\mathsf{VAR}(\hat{\phi})$ of MLE $\hat{\phi}$
is independent of $\boldsymbol{\beta}$, i.e.,
\begin{equation}
\dfrac{\partial\mathsf{VAR}(\ell_{\phi}^{\prime})}{\partial\beta_{k}}=0,\quad\forall k\in\{1,2,3\},\label{eq:Independence}
\end{equation}
we can focus on determinant $|I_{3\times3}(\boldsymbol{\beta})|$,
as if the dispersion parameter $\phi$ would be known. Note that we
do not need a fourth stimulus to be able to estimate parameter $\phi$.
The reason is that the MLE of $\boldsymbol{\beta}$ can be obtained
independently of the estimation of $\phi$ because the likelihood
equations for $\boldsymbol{\beta}$ do not involve the parameter $\phi$,
see (\ref{eq:score-functions}). Additionally, $\phi$ can be estimated
from the variability in the observations $y_{ij}$ if $n>1$, since
$\mathsf{VAR}(y_{ij}|x_{i})=b''(\theta_{i})a(\phi)$ and $\theta_{i}$
can be estimated with MLE $\hat{\boldsymbol{\beta}}$ and $x_{i}$.

Condition (\ref{eq:Independence}) is satisfied for exponential families
of distributions of the form $f(y|\theta,\eta)=\exp\{\eta[y\theta_{i}-b(\theta_{i})]+d_{1}(y)+d_{2}(\eta)+\eta c(y)\}$,
where $\eta=1/a(\phi)$, since the derivative $\partial\ell_{\eta}^{\prime}/\partial\eta$
of the score function $\ell_{\eta}^{\prime}$ is independent of $y$
and $\boldsymbol{\beta}$ (see formula (1.22) of \cite{key-23} on
page 8). This condition (\ref{eq:Independence}) holds for all well-known
distribution functions that will be used in this study which are Poisson, Binomial, Negative Binomial, Gaussian,
Gamma, and inverse Gaussian (see Table 1.1 of \cite{key-23}).

We are interested in the smallest number of stimuli that would maximize
determinant $|I_{4\times4}(\boldsymbol{\beta},\phi)|$, i.e. the locally
minimal D-optimality criterion \cite{key-21}. Assuming that condition
(\ref{eq:Independence}) holds true, we can focus on only three stimuli
$x_{1}$, $x_{2}$, and $x_{3}$, since determinant $|I_{3\times3}(\boldsymbol{\beta})|$
contains only three parameters. Thus we are looking for stimuli $x_{1}$,
$x_{2}$, and $x_{3}$, with $x_{1}<x_{2}<x_{3}$, such that
\begin{equation}
\underset{L\leq x_{1}<x_{2}<x_{3}\leq U}{\mathrm{arg\:max}}|I_{3\times3}(\boldsymbol{\beta})|,
\end{equation}
with $L\geq0$ and $U<\infty$ a known lower and upper bound on the
range of stimuli, respectively, typically determined by practical
limitations. With the help of Matlab we were able to express the determinant
$|I_{3\times3}(\boldsymbol{\beta})|$ in an explicit form equal to
\begin{equation}
\beta_{2}^{2}\left[(x_{1}x_{2})^{\beta_{3}}\log(\tfrac{x_{2}}{x_{1}})-(x_{1}x_{3})^{\beta_{3}}\log(\tfrac{x_{3}}{x_{1}})+(x_{2}x_{3})^{\beta_{3}}\log(\tfrac{x_{3}}{x_{2}})\right]^{2}\prod_{i=1}^{3}[n_{i}g'(\mu_{i})].\label{eq:D}
\end{equation}
It is important to realize that the sample sizes $n_{1}$, $n_{2}$,
and $n_{3}$ do not influence the choice of stimuli $x_{1}$, $x_{2}$,
and $x_{3}$ for maximization of (\ref{eq:D}), since only the product
$n_{1}n_{2}n_{3}$ is involved in (\ref{eq:D}). Thus if the optimal
design is known and the total sample size $n=n_{1}+n_{2}+n_{3}$ is
determined, it would be best to choose the same sample size in each
stimulus to maximize precision.

\section{D-optimal designs}

Here we will focus on finding the optimal values for $x_{1}$, $x_{2}$,
and $x_{3}$ that would maximize determinant $|I_{3\times3}(\boldsymbol{\beta})|$
in (\ref{eq:D}) under constraint $L\leq x_{1}<x_{2}<x_{3}\leq U$,
with $L\geq0$ and $U<\infty$. We will see that the choice of the
three stimuli depends on the mathematical behavior of the link function
$g$. Note that our results will be D-optimal when either $\phi$
is known or otherwise when condition (\ref{eq:Independence}) is satisfied.
Our main results are formulated in the following three theorems. The
proofs are provided in Appendix B.
\begin{thm}
\label{th-x1}If $g'(\mu)\geq0$, and $g''(\mu)\leq0$ holds, then
the optimal stimulus $x_{1}^{\mathrm{opt}}$ for $x_{1}$ that maximizes
determinant $|I_{3\times3}(\boldsymbol{\beta})|$ in (\ref{eq:D}),
is the smallest possible stimulus value, i.e., $x_{1}^{\mathrm{opt}}=L$.
\end{thm}
\begin{proof}
See Appendix B.
\end{proof}
The two conditions on the link function in Theorem \ref{th-x1} indicate
that we are studying concave increasing link functions. This is satisfied
for the identity link function $g(\mu)=\mu$ with $\mu\in\mathbb{R}$,
the log link function $g(\mu)=\log(\mu)$ with $\mu\in(0,\infty)$,
the square root link function $g(\mu)=\sqrt{\mu}$ with $\mu\in(0,\infty)$
, the negative inverse link function $g(\mu)=-\mu^{-1}$ with $\mu\in(0,\infty)$,
and the half negative inverse-square link function $g(\mu)=-0.5\mu^{-2}$
with $\mu\in(0,\infty)$. Thus for these link functions we need to
choose the first stimulus $x_{1}$ as small as possible if we want
to maximize the determinant of the Fisher information matrix. For
the logit link function $g(\mu)=\log(\mu/[N-\mu]$) with $\mu\in(0,N)$
the condition $g''(\mu)\leq0$ is only guaranteed when $\mu\leq N/2$.
Thus when $\mu>N/2$, we do not know if stimulus $x_{1}$ should be
selected as small as possible. 
\begin{thm}
\label{th-x3}If $g'(\mu)\geq0$, $g''(\mu)\leq0$, and $g''(\mu)\mu+2g'(\mu)\geq0$
holds, then the optimal stimulus $x_{3}^{\mathrm{opt}}$ for $x_{3}$
that maximizes determinant $|I_{3\times3}(\boldsymbol{\beta})|$ in
(\ref{eq:D}), is the largest possible value, i.e., $x_{3}^{\mathrm{opt}}=U$.
\end{thm}
\begin{proof}
See Appendix B
\end{proof}
The third condition $g''(\mu)\mu+2g'(\mu)\geq0$ in Theorem \ref{th-x3}
for link function $g$ would be satisfied for most of the link functions
(e.g., $g(\mu)=\mu$, $g(\mu)=\log(\mu)$, $g(\mu)=\sqrt{\mu}$, $g(\mu)=\log(\mu/[N-\mu])$,
and $g(\mu)=-\mu^{-1}$), but it does not hold for the canonical link
function $g(\mu)=-0.5\mu^{-2}$, with $\mu\in(0,\infty)$, for the
inverse Gaussian distribution. Recall that the canonical link function
$g(\mu)=\log(\mu/[N-\mu])$ of the Binomial distribution satisfies
condition $g''(\mu)\leq0$ only when $\mu\leq N/2$. Thus the third
stimulus $x_{3}$ should be chosen as large as possible for most canonical
link functions, but for the inverse Gaussian and Binomial distribution
with their canonical link function it may be possible to obtain better
designs when we stay away from the boundary value $U$ (see Section
\ref{subsec:Examples}).
\begin{thm}
\label{th-x2}Assume that $g'(\mu)\geq0$, and $g''(\mu)\leq0$ holds
and let $x_{1}$ and $x_{3}$ be given stimuli, then the optimal stimulus
$x_{2}^{\mathrm{opt}}$ for $x_{2}$ that maximizes determinant $|I_{3\times3}(\boldsymbol{\beta})|$
in (\ref{eq:D}), is obtained by solving the following equation
\begin{equation}
\begin{split}\beta_{2}\beta_{3}g''\left(\mu_{2}\right)\left[(x_{1}x_{2})^{\beta_{3}}\log(\tfrac{x_{2}}{x_{1}})-(x_{1}x_{3})^{\beta_{3}}\log(\tfrac{x_{3}}{x_{1}})+(x_{2}x_{3})^{\beta_{3}}\log(\tfrac{x_{3}}{x_{2}})\right]\\
+2g'\left(\mu_{2}\right)\left[\beta_{3}x_{1}^{\beta_{3}}\log(\tfrac{x_{2}}{x_{1}})+\beta_{3}x_{3}^{\beta_{3}}\log(\tfrac{x_{3}}{x_{2}})+x_{1}^{\beta_{3}}-x_{3}^{\beta_{3}}\right]=0.
\end{split}
\label{eq:optimal-x2}
\end{equation}
The optimal solution $x_{2}^{\mathrm{opt}}$ is an element of interval
$(x_{1},x_{3})$ and satisfies constraint $\beta_{3}\log(x_{2}^{\mathrm{opt}})\leq[x_{3}^{\beta_{3}}\log(x_{3}^{\beta_{3}})-x_{1}^{\beta_{3}}\log(x_{1}^{\beta_{3}})-(x_{3}^{\beta_{3}}-x_{1}^{\beta_{3}})]/[x_{3}^{\beta_{3}}-x_{1}^{\beta_{3}}]$.
\end{thm}
\begin{proof}
See Appendix B.
\end{proof}
Theorems \ref{th-x1}, \ref{th-x3}, and \ref{th-x2} all formulate
optimal choices of only one stimulus, conditionally on the other two
stimuli, whether these other two stimuli are chosen optimally or not.
The theorems tell us what to do with this one stimulus to maximize
determinant $|I_{3\times3}(\boldsymbol{\beta})|$ when the other stimuli
are already provided. For Theorem \ref{th-x2} it shows that the best
choice for $x_{2}\in(x_{1},x_{3})$ is the value that solves equation
(\ref{eq:optimal-x2}), if we wish to maximize determinant $|I_{3\times3}(\boldsymbol{\beta})|$.
\begin{cor}
\label{th_x2_x1is0}If the conditions of Theorem \ref{th-x2} hold
and $x_{1}=0$, then the optimal stimulus $x_{2}^{\mathrm{opt}}$
for $x_{2}$ that maximizes determinant $|I_{3\times3}(\boldsymbol{\beta})|$
in (\ref{eq:D}), is obtained by solving equation
\begin{equation}
\begin{split}\beta_{2}\beta_{3}g''\left(\mu_{2}\right)x_{2}^{\beta_{3}}\log(\tfrac{x_{3}}{x_{2}})+2g'\left(\mu_{2}\right)\left[\beta_{3}\log(\tfrac{x_{3}}{x_{2}})-1\right]=0\end{split}
\label{eq:optimal-x2-2}
\end{equation}
and the optimal solution satisfies $0<x_{2}^{\mathrm{opt}}\leq x_{3}\exp\{-\beta_{3}^{-1}\}$.
\end{cor}
\begin{proof}
Substituting $x_{1}=0$ in (\ref{eq:optimal-x2}) leads directly to
equation (\ref{eq:optimal-x2-2}), since the value of the third stimulus
is always larger than zero (i.e., $x_{3}>0$) to guarantee that we
have three different stimuli. Furthermore, substituting $x_{1}=0$
in the boundaries on $x_{2}^{\mathrm{opt}}$ results in the lower
and upper boundary $0$ and $x_{3}\exp\{-\beta_{3}^{-1}\}$, respectively.
\end{proof}
Corollary \ref{th_x2_x1is0} and Theorem \ref{th-x2} demonstrate
that the optimal value $x_{2}^{\mathrm{opt}}$ for the second stimulus
depends on the two other stimuli $x_{1}$ and $x_{3}$, and on the
link function $g$ through its derivatives $g'$ and $g''$. Equations
(\ref{eq:optimal-x2}) and (\ref{eq:optimal-x2-2}) also show that
distributions with the same link function result in the same D-optimal
design. Thus the D-optimal designs for a Poisson and Negative Binomial
distributed response $y$ are identical, since they both have canonical
link function $g(\mu)=\log(\mu)$ and they just differ in the dispersion
variable $\phi$. Whether all three model parameters $\boldsymbol{\beta}$
are involved in $x_{2}^{\mathrm{opt}}$, depends on the link function
(see Section \ref{subsec:Examples}), but it always involves the power
parameter $\beta_{3}$.

\subsection{Examples\label{subsec:Examples}}

We will discuss D-optimal designs for the well-known distribution
functions of the exponential family of distributions using their canonical
link function. As illustration we will assume a measurement system
analysis for which the stimuli can range from $L=0$ to $U=15$. We
will consider six combinations of parameter settings for $\beta_{1}\in\{0.5,1.0\}$,
$\beta_{2}\in\{0.8,1.0,1.2\}$, and $\beta_{3}\in\{0.9,1.0,1.1\}$.

\textbf{\uline{Gaussian Distribution:}} For the Gaussian distribution
with the identity link function, equation (\ref{eq:optimal-x2}) can
be solved explicitly. Theorems \ref{th-x1} and \ref{th-x3} imply
that the first and third optimal stimulus should be chosen equal to
$x_{1}^{\mathrm{opt}}=L$ and $x_{3}^{\mathrm{opt}}=U$, respectively.
Then the optimal second stimulus is equal to
\begin{equation}
x_{2}^{\mathrm{opt}}=\exp\left\{ [U^{\beta_{3}}\log(U)-L^{\beta_{3}}\log(L)]/[U^{\beta_{3}}-L^{\beta_{3}}]-\beta_{3}^{-1}\right\} ,\label{eq:opt-x2-normal}
\end{equation}
which depends only on the power parameter $\beta_{3}$ (and not on
the intercept $\beta_{1}$ and slope $\beta_{2}$). In case the lower
boundary $L$ is equal to zero, the optimal second stimulus reduces
to $x_{2}^{\mathrm{opt}}=U\exp\{-\beta_{3}^{-1}\}$, which is equal
to the upper bound on $x_{2}^{\mathrm{opt}}$ mentioned in Corollary
\ref{th_x2_x1is0}. Table \ref{tab:D-Rest} shows the optimal stimulus
$x_{2}^{\mathrm{opt}}$ for our illustration. Since $\beta_{3}\approx1$,
the stimulus is approximately $36.8\%$ of the upper boundary $U=15$.

\begin{table}[h]
\caption{Optimal value for the second stimulus $x_{2}$ for different distributions
of the exponential family ($L=0$ and $U=15$).\label{tab:D-Rest}}

\centering

{\small{}}%
\begin{tabular}{|c|c|c||c|c|c|c|c|c|}
\hline 
\multicolumn{3}{|c||}{{\small{}Parameters}} & \multirow{2}{*}{{\small{}Gaussian}} & \multirow{2}{*}{{\small{}Poisson}} & \multirow{2}{*}{{\small{}Gamma}} & \multicolumn{3}{c|}{{\small{}Binomial}}\tabularnewline
\cline{1-3} \cline{2-3} \cline{3-3} \cline{7-9} \cline{8-9} \cline{9-9} 
{\small{}$\beta_{1}$} & {\small{}$\beta_{2}$} & {\small{}$\beta_{3}$} &  &  &  & {\small{}$N=25$} & {\small{}$N=50$} & {\small{}$N=100$}\tabularnewline
\hline 
\hline 
{\small{}0.5} & {\small{}1.2} & {\small{}0.9} & {\small{}4.94} & {\small{}2.24} & {\small{}0.70} & {\small{}2.65} & {\small{}2.41} & {\small{}2.32}\tabularnewline
\hline 
{\small{}0.5} & {\small{}1.0} & {\small{}1.0} & {\small{}5.52} & {\small{}2.67} & {\small{}0.90} & {\small{}3.16} & {\small{}2.87} & {\small{}2.76}\tabularnewline
\hline 
{\small{}0.5} & {\small{}0.8} & {\small{}1.1} & {\small{}6.04} & {\small{}3.10} & {\small{}1.14} & {\small{}3.66} & {\small{}3.33} & {\small{}3.20}\tabularnewline
\hline 
{\small{}1.0} & {\small{}1.2} & {\small{}0.9} & {\small{}4.94} & {\small{}2.58} & {\small{}1.12} & {\small{}3.04} & {\small{}2.77} & {\small{}2.67}\tabularnewline
\hline 
{\small{}1.0} & {\small{}1.0} & {\small{}1.0} & {\small{}5.52} & {\small{}3.02} & {\small{}1.38} & {\small{}3.57} & {\small{}3.25} & {\small{}3.13}\tabularnewline
\hline 
{\small{}1.0} & {\small{}0.8} & {\small{}1.1} & {\small{}6.04} & {\small{}3.47} & {\small{}1.68} & {\small{}4.08} & {\small{}3.71} & {\small{}3.58}\tabularnewline
\hline 
\end{tabular}{\small\par}
\end{table}

\textbf{\uline{Poisson and Negative Binomial Distribution:}} The
canonical link function is $g(\mu)=\log(\mu)$, which implies that
$x_{1}^{\mathrm{opt}}=L$ and $x_{3}^{\mathrm{opt}}=U$ (Theorems
\ref{th-x1} and \ref{th-x3}, respectively). A solution for equation
(\ref{eq:optimal-x2}) can only be obtained numerically and this equation
contains all three parameters of the Mitscherlich function. When $L=0$,
equation (\ref{eq:optimal-x2-2}) reduces to
\begin{equation}
\left[\beta_{1}+\beta_{2}x_{2}^{\beta_{3}}\right]\left[\beta_{3}\log(\tfrac{U}{x_{2}})-2\right]+\beta_{1}\beta_{3}\log(\tfrac{U}{x_{2}})=0,\label{eq:Poisson}
\end{equation}
which can not be solved explicitly either and still depends on all
three parameters. However, the optimal solution for the second stimulus
$x_{2}^{\mathrm{opt}}$ is inside the interval $[U\exp\{-2\beta_{3}^{-1}\},U\exp\{-\beta_{3}^{-1}\}]$.
Indeed, the left-hand side in (\ref{eq:Poisson}) is non-negative
when $\beta_{3}\log(U/x_{2})\ge2$, which means that $x_{2}^{\mathrm{opt}}\in[U\exp\{-2\beta_{3}^{-1}\},U)$.
In addition, the left-hand side in (\ref{eq:Poisson}) is non-positive
if $\beta_{3}\log(U/x_{2})\le1$, which means that $x_{2}^{\mathrm{opt}}\in(0,U\exp\{-\beta_{3}^{-1}\}]$,
but this was already known from the upper boundary in Corollary \ref{th_x2_x1is0}.
When the intercept $\beta_{1}=0$, the optimal value for the second
stimulus becomes $x_{2}^{\mathrm{opt}}=U\exp\{-2\beta_{3}^{-1}\}$,
which is different from the solution of the Gaussian distribution.
Table \ref{tab:D-Rest} shows that the optimal stimulus $x_{2}^{\mathrm{opt}}$
is substantially lower than the solution of the Gaussian distribution
with the identity link function.

\textbf{\uline{Gamma Distribution:}} The canonical link function
$g(\mu)=-\mu^{-1}$ together with Theorems \ref{th-x1} and \ref{th-x3},
imply that $x_{1}^{\mathrm{opt}}=L$ and $x_{3}^{\mathrm{opt}}=U$.
The solution of equation (\ref{eq:optimal-x2}) can be determined
numerically and it involves all three parameters $\beta_{1}$, $\beta_{2}$,
and $\beta_{3}$ of the Mitscherlich function. If we assume again
that $L=0$, equation (\ref{eq:optimal-x2-2}) reduces to
\begin{equation}
\beta_{1}+\beta_{2}x_{2}^{\beta_{3}}+\beta_{1}\beta_{3}\log(x_{2})=\beta_{1}\beta_{3}\log(U),
\end{equation}
which can be solved numerically for different values of $\beta_{1}$,
$\beta_{2}$, and $\beta_{3}$. In case $\beta_{1}=0$, the optimal
second stimulus becomes equal to $x_{2}^{\mathrm{opt}}=0$, but these
results are not allowed for a gamma distribution with a positive range.
The parameter $\beta_{1}$ should be positive when we allow the stimulus
$x_{1}$ to be equal to zero. Table \ref{tab:D-Rest} shows that the
optimal stimulus $x_{2}^{\mathrm{opt}}$ is still close to zero when
$\beta_{1}>0$.

\textbf{\uline{Binomial Distribution:}} Solving equation (\ref{eq:optimal-x2})
or even (\ref{eq:optimal-x2-2}) for the Binomial distribution with
canonical link function $g(\mu)=\log(\mu/[N-\mu])$ is very tedious
and does not easily reduce into manageable functions. The solution
$x_{2}^{\mathrm{opt}}$ depends on all three parameters $\beta_{1}$,
$\beta_{2}$, and $\beta_{3}$ of the Mitscherlich function and numerical
approaches should be used to determine the D-optimal design. When
$\beta_{1}+\beta_{2}U^{\beta_{3}}\leq N/2$ we know that $x_{1}^{\mathrm{opt}}=L$
and $x_{3}^{\mathrm{opt}}=U$ based on Theorems \ref{th-x1} and \ref{th-x3}.
Thus for our illustration with $L=0$ and $U=15$ and the six selected
combinations of parameter settings $\beta_{1}\in\{0.5,1.0\}$, $\beta_{2}\in\{0.8,1.0,1.2\}$,
and $\beta_{3}\in\{0.9,1.0,1.1\}$ in Table \ref{tab:D-Rest}, a sample
size of $N=34$ would be enough to satisfy the condition $g''(\mu)\leq0$
in Theorems \ref{th-x1}, \ref{th-x3} and \ref{th-x2}. Thus for
$N=50$ and $N=100$, we know that $x_{1}^{\mathrm{opt}}=L=0$ and
$x_{3}^{\mathrm{opt}}=U=15$, but for $N=25$ we do not know this.
Investigating the optimal stimuli $x_{i}\in[0,15]$ using numerical
calculations (just calculating the determinant for a grid of stimuli
using step size $h=0.01$) still shows that $x_{1}^{\mathrm{opt}}=0$
and $x_{3}^{\mathrm{opt}}=15$. Table \ref{tab:D-Rest} shows the
results of the optimal stimulus $x_{2}^{\mathrm{opt}}$ for all settings.
The results show that $x_{2}^{\mathrm{opt}}$ is close to the optimal
value of the Poisson, which is not a surprise since the Binomial and
Poisson distribution are very similar, in particular when the sample
size $N$ is increasing. This resemblence between the two distributions
may explain why the optimal stimulus $x_{1}^{\mathrm{opt}}$ and $x_{3}^{\mathrm{opt}}$
are still equal to $L=0$ and $U=15$ when $N=25$.

\textbf{\uline{Inverse Gaussian Distribution:}} Since the canonical
link function $g(\mu)=-0.5\mu^{-2}$ does not satisfy the conditions
of Theorem \ref{th-x3}, i.e., $\mu g''(\mu)+2g'(\mu)=-\mu^{-3}<0$
for $\mu\in(0,\infty)$, it is not known how to choose the optimal
value for the third stimulus $x_{3}$, but we know from Theorem \ref{th-x1}
that $x_{1}^{\mathrm{opt}}=L$. If we assume that $L=0$, determinant
$|I_{3\times3}(\boldsymbol{\beta})|$ becomes equal to
\begin{equation}
n_{1}n_{2}n_{3}\beta_{1}^{-3}\beta_{2}^{2}\left[(x_{2}x_{3})^{\beta_{3}}\log(\tfrac{x_{3}}{x_{2}})\right]^{2}\left[\beta_{1}+\beta_{2}x_{2}^{\beta_{3}}\right]^{-3}\left[\beta_{1}+\beta_{2}x_{3}^{\beta_{3}}\right]^{-3}.\label{eq:D-IG}
\end{equation}
Given the parameters $\beta_{1}$, $\beta_{2}$ and $\beta_{3}$ a
grid search for $x_{2}$ and $x_{3}$ can be conducted to maximize
(\ref{eq:D-IG}). Table \ref{tab:D-IG} shows the optimal stimuli
for a measurement study where the stimuli can range from $L=0$ to
$U=15$. We used a step size of $0.01$ in our grid search. Table
\ref{tab:D-IG} shows that $x_{3}^{\mathrm{opt}}$ is far away from
the boundary $U=15$ and $x_{2}^{\mathrm{opt}}$ is relatively close
to $L=0$.

\begin{table}[h]
\caption{D-optimal design for the Inverse Gaussian distribution with canonical
link function $g(\mu)=-0.5\mu^{-2}$ and boundaries $L=0$ and $U=15$.\label{tab:D-IG}}

\centering

{\small{}}%
\begin{tabular}{|c|c|c||c|c|c|c||c|c|c|}
\hline 
\multicolumn{3}{|c||}{{\small{}Parameters}} & \multicolumn{4}{c||}{{\small{}Optimal Design{*}}} & \multicolumn{3}{c|}{{\small{}Equidistant Designs}}\tabularnewline
\hline 
\hline 
{\small{}$\beta_{1}$} & {\small{}$\beta_{2}$} & {\small{}$\beta_{3}$} & {\small{}$x_{1}^{\mathrm{opt}}$} & {\small{}$x_{2}^{\mathrm{opt}}$} & {\small{}$x_{3}^{\mathrm{opt}}$} & {\small{}$|I_{3\times3}(\boldsymbol{\beta})|$} & {\small{}$d=60$} & {\small{}$d=30$} & {\small{}$d=15$}\tabularnewline
\hline 
{\small{}0.5} & {\small{}1.2} & {\small{}0.9} & {\small{}0} & {\small{}0.26} & {\small{}5.21} & {\small{}1.455} & {\small{}70.8\%} & {\small{}55.4\%} & {\small{}21.0\%}\tabularnewline
\hline 
{\small{}0.5} & {\small{}1.0} & {\small{}1.0} & {\small{}0} & {\small{}0.36} & {\small{}5.32} & {\small{}1.697} & {\small{}64.5\%} & {\small{}65.2\%} & {\small{}32.3\%}\tabularnewline
\hline 
{\small{}0.5} & {\small{}0.8} & {\small{}1.1} & {\small{}0} & {\small{}0.48} & {\small{}5.58} & {\small{}2.192} & {\small{}51.8\%} & {\small{}68.8\%} & {\small{}45.3\%}\tabularnewline
\hline 
{\small{}1.0} & {\small{}1.2} & {\small{}0.9} & {\small{}0} & {\small{}0.57} & {\small{}11.34} & {\small{}0.045} & {\small{}64.9\%} & {\small{}73.4\%} & {\small{}45.9\%}\tabularnewline
\hline 
{\small{}1.0} & {\small{}1.0} & {\small{}1.0} & {\small{}0} & {\small{}0.72} & {\small{}10.65} & {\small{}0.053} & {\small{}51.3\%} & {\small{}73.7\%} & {\small{}59.4\%}\tabularnewline
\hline 
{\small{}1.0} & {\small{}0.8} & {\small{}1.1} & {\small{}0} & {\small{}0.91} & {\small{}10.53} & {\small{}0.068} & {\small{}36.0\%} & {\small{}66.4\%} & {\small{}69.8\%}\tabularnewline
\hline 
\end{tabular}{\small\par}
\raggedright{}{\footnotesize{}{*}We have assumed that the sample sizes
are equal to one ($n_{1}=n_{2}=n_{3}=1$) for calculation of $I_{3\times3}(\boldsymbol{\beta})$.}{\footnotesize\par}
\end{table}

The optimal solution for the second stimulus should satisfy equation
(\ref{eq:optimal-x2-2}), which can be rewritten in
\begin{equation}
\beta_{3}\left[\beta_{2}x_{2}^{\beta_{3}}-2\beta_{1}\right]\log(\tfrac{x_{3}}{x_{2}})+2\left[\beta_{1}+\beta_{2}x_{2}^{\beta_{3}}\right]=0.\label{eq:IG-x2}
\end{equation}
In case $\beta_{2}x_{2}^{\beta_{3}}-2\beta_{1}>0$ the left-hand side
in (\ref{eq:IG-x2}) is positive, while it is negative when $x_{2}$
gets close to zero. This implies that $x_{2}^{\mathrm{opt}}\in(0,[2\beta_{1}/\beta_{2}]^{1/\beta_{3}})$,
illustrating that $x_{2}^{\mathrm{opt}}$ can never be far away from
zero (unless $\beta_{2}$ is close to zero). This upper bound on $x_{2}$
can be useful in a grid search for maximization of (\ref{eq:D-IG}),
since $x_{2}$ should never go beyond $[2\beta_{1}/\beta_{2}]^{1/\beta_{3}}$
and $x_{3}$ should never start before $[2\beta_{1}/\beta_{2}]^{1/\beta_{3}}$
when $x_{2}^{\mathrm{opt}}$ reaches this bound.

Furthermore, in a measurement system analysis it is common to use
equidistant stimuli designs, either in the original scale or otherwise
in the logarithmic scale. Considering the D-optimal design in Table
\ref{tab:D-IG}, an equidistant design in the logarithmic scale is
closer to the D-optimal design than an equidistant design in its original
scale, although the dilution factor varies with the parameters $\beta_{1}$,
$\beta_{2}$ and $\beta_{3}$. The efficiency of these so-called dilution
designs with respect to the optimal design is provided in Table \ref{tab:D-IG}
for different dilution factors $d$ (and taking $x_{3}=U$). It is
obvious that our D-optimal design is substantially more efficient
than a dilution experiment.

\section{Relations to Earlier Work and Extensions\label{sec:Extensions}}

As we mentioned earlier, D-optimal designs for the Mitscherlich function
were already obtained for the normal distribution with homogeneous
residual variances \cite{key-1,key-5,key-41}, but they used different
parametrizations. Here we will show that these parametrizations are
irrelevant. We will also show that under certain conditions our D-optimal
solution can be obtained from minimizing a weighted least squares.
However, when we start considering heteroscedastic residual variances,
these two approaches for optimal designs can also be different. We
finish with a discussion on extending our work to transformations
of the Mitscherlich function $\psi(\beta_{1}+\beta_{2}x_{i}^{\beta_{3}})$.

\subsection{Existing D-Optimal Designs}

In our results we formulated the Mitscherlich non-linear function
as $\mathbb{E}(y_{ij}|x_{i})=\beta_{1}+\beta_{2}x_{i}^{\beta_{3}}$,
with $x_{i}$ a non-negative stimulus, and $\beta_{2}$ and $\beta_{3}$
both positive. However, Box and Lucas \cite{key-1} introduced the
Mitscherlich function as $\mathbb{E}(y_{ij}|z_{i})=\beta_{1}-\beta_{2}\exp\{-\beta_{3}z_{i}\}$,
with $\beta_{2}>0$ and $\beta_{3}>0$, Han and Chaloner \cite{key-5}
used $\mathbb{E}(y_{ij}|z_{i})=\beta_{1}+\beta_{2}\exp\{-\beta_{3}z_{i}\}$,
with $\beta_{2}>0$ and $\beta_{3}>0$, and Dette \textit{et al}.,
\cite{key-41} used $\mathbb{E}(y_{ij}|z_{i})=\beta_{1}+\beta_{2}\exp\{z_{i}/\tilde{\beta}_{3}\}$,
with $\beta_{2}>0$ and $\tilde{\beta}_{3}>0$. It is reasonably straightforward
to calculate the D-optimal designs for these three formulations from
our results , since we just deal with reformulations of our own form
$\mathbb{E}(y_{ij}|x_{i})=\beta_{1}+\beta_{2}x_{i}^{\beta_{3}}$ with
$\beta_{2}>0$ and $\beta_{3}>0$.

For Dette \textit{et al}.'s formulation under their assumption of
normality with homoscedastic residual variances, we need to both reparametrize
$\beta_{3}$ and transform the stimulus.Our stimulus $x_{i}$ can
be taken equal to $x_{i}=\exp\{z_{i}$\} and the power parameter $\beta_{3}$
can be taken equal to $\beta_{3}=1/\tilde{\beta}_{3}$. Here the transformation
for the stimulus is an increasing function of $z_{i}$, thus Theorems
\ref{th-x1} and \ref{th-x3} indicate we need to choose $z_{1}^{\mathrm{opt}}$
as small as possible and $z_{3}^{\mathrm{opt}}$ as large as possible.
In case $z_{i}$ represents time, the smallest value could potentially
be zero, implying that $x_{1}^{\mathrm{opt}}\geq1$. Using optimal
solution (\ref{eq:opt-x2-normal}) with $L=\exp\{z_{1}^{\mathrm{opt}}\}$,
$U=\exp\{z_{3}^{\mathrm{opt}}\}$, $\beta_{3}=1/\tilde{\beta}_{3}$,
and $x_{2}^{\mathrm{opt}}=\exp\{z_{2}^{\mathrm{opt}}\}$, the optimal
solution $z_{2}^{\mathrm{opt}}$ is now equal to
\begin{equation}
z_{2}^{\mathrm{opt}}=\dfrac{z_{3}^{\mathrm{opt}}\exp\{z_{3}^{\mathrm{opt}}/\tilde{\beta}_{3}\}-z_{1}^{\mathrm{opt}}\exp\{z_{1}^{\mathrm{opt}}/\tilde{\beta}_{3}\}}{\exp\{z_{3}^{\mathrm{opt}}/\tilde{\beta}_{3}\}-\exp\{z_{1}^{\mathrm{opt}}/\tilde{\beta}_{3}\}}-\tilde{\beta}_{3},\label{eq:z2-opt}
\end{equation}
which was indeed presented in \cite{key-41}.

The optimal solutions for the other two formulations can be obtained
in a similar way. For the formulation of Box and Lucas, we can take
$x_{i}=$exp$\{z_{i}\}$. If one realizes that our proofs of Theorems
\ref{th-x1}, \ref{th-x3}, and \ref{th-x2} remain correct if both
$\beta_{2}$ and $\beta_{3}$ become negative (instead of being both
positive), we obtain also $z_{1}^{\mathrm{opt}}=z_{\mathrm{min}}$,
$z_{3}^{\mathrm{opt}}=z_{\mathrm{max}}$, with $z_{\mathrm{min}}$
and $z_{\mathrm{max}}$ the minimal and maximal allowable value for
stimulus $z$, and $z_{2}^{\mathrm{opt}}$ satisfies (\ref{eq:z2-opt})
with $\tilde{\beta}_{3}$ replaced by $-\beta_{3}^{-1}$, which was
obtained by \cite{key-1}. For the formulation of Han and Chaloner,
with $x_{i}=$exp$\{-z_{i}\}$, we obtain $z_{1}^{\mathrm{opt}}=z_{\mathrm{max}}$,
$z_{3}^{\mathrm{opt}}=z_{\mathrm{min}}$, and
\[
z_{2}^{\mathrm{opt}}=\dfrac{z_{1}^{\mathrm{opt}}\exp\{-\beta_{3}z_{1}^{\mathrm{opt}}\}-z_{3}^{\mathrm{opt}}\exp\{-\beta_{3}z_{3}^{\mathrm{opt}}\}}{\exp\{-\beta_{3}z_{1}^{\mathrm{opt}}\}-\exp\{-\beta_{3}z_{3}^{\mathrm{opt}}\}}+\dfrac{1}{\beta_{3}},
\]
which was reported by \cite{key-5}. The order for $z_{1}^{\mathrm{opt}}$
and $z_{3}^{\mathrm{opt}}$ is changed, since the stimulus $x_{i}=\exp\{-z_{i}\}$
is now a decreasing function of $z_{i}$.

\subsection{Weighted Least Squares and Heteroscedasticity}

The optimal design for the Mitscherlich function that was proposed
by Box and Lucas \cite{key-1} in 1959, was based on the linearization
of the Mitscherlich non-linear function and the maximization of the
determinant of the corresponding design matrix (as if they were constructing
a D-optimal design for a linear regression problem \cite{key-52}).
Their design matrix was equal to
\begin{equation}
\boldsymbol{X}=\begin{pmatrix}\frac{\partial\mu_{1}}{\partial\beta_{1}} & \frac{\partial\mu_{1}}{\partial\beta_{2}} & \frac{\partial\mu_{1}}{\partial\beta_{3}}\\
\frac{\partial\mu_{2}}{\partial\beta_{2}} & \frac{\partial\mu_{2}}{\partial\beta_{2}} & \frac{\partial\mu_{2}}{\partial\beta_{3}}\\
\frac{\partial\mu_{3}}{\partial\beta_{1}} & \frac{\partial\mu_{3}}{\partial\beta_{2}} & \frac{\partial\mu_{3}}{\partial\beta_{3}}
\end{pmatrix}=\begin{pmatrix}1 & x_{1}^{\beta_{3}} & \beta_{2}x_{1}^{\beta_{3}}\log\left(x_{1}\right)\\
1 & x_{2}^{\beta_{3}} & \beta_{2}x_{2}^{\beta_{3}}\log\left(x_{2}\right)\\
1 & x_{3}^{\beta_{3}} & \beta_{2}x_{3}^{\beta_{3}}\log\left(x_{3}\right)
\end{pmatrix}\label{f.1}
\end{equation}
and they maximized determinant $|\boldsymbol{X}^{T}\boldsymbol{X}|$
over $L\leq x_{1}<x_{2}<x_{3}\leq U$, assuming that at each stimulus
the same sample size was used. For an imbalanced design the determinant
becomes $|\boldsymbol{X}^{T}\boldsymbol{W}\boldsymbol{X}|,$ with
$\boldsymbol{W}$ a $3\times3$ diagonal matrix with $n_{1}$, $n_{2}$,
and $n_{3}$ at the diagonal. Under the assumption of normality with
the identity link function, this linearization with an imbalanced
design leads to the maximization of $|I_{3\times3}(\boldsymbol{\beta})|$
in (\ref{eq:D}), see also the variances of the score functions in
(\ref{eq:Fisher-Information}). However, for other distributions,
with another canonical link function than the identity, our D-optimal
design would deviate from the optimal design of \cite{key-1} since
the linearization should involve the link function.

To generalize the approach of \cite{key-1}, we should change $\boldsymbol{X}^{T}\boldsymbol{X}$
such that it represents the variances in (\ref{eq:Fisher-Information}).
Thus if we would choose weight matrix $\boldsymbol{W}$ by
\begin{equation}
\boldsymbol{W}=\begin{pmatrix}n_{1}g'(\mu_{1})/a(\phi) & 0 & 0\\
0 & n_{2}g'(\mu_{2})/a(\phi) & 0\\
0 & 0 & n_{3}g'(\mu_{3})/a(\phi)
\end{pmatrix}\label{eq:weights}
\end{equation}
and maximize determinant $|\boldsymbol{X}^{T}\boldsymbol{W}\boldsymbol{X}|$,
we would maximize determinant $|I_{3\times3}(\boldsymbol{\beta})|$
in (\ref{eq:D}). Thus by changing the least squares approach of \cite{key-1}
to a weighted least squares approach, we obtain the D-optimal designs
for the Mitscherlich non-linear function for any of the distributions
in the exponential family that satisfy condition (\ref{eq:Independence}).
It should be noted that the weight $w_{i}=n_{i}a^{-1}(\phi)g'(\mu_{i})$
in (\ref{eq:weights}) is equal to $w_{i}=n_{i}[\mathsf{VAR}(y_{ij})]^{-1}$,
the typical weights used in a weighted linear regression approach.

If we return to the normal distribution again and assume that the
dispersion parameter $\phi$ depends on the mean $\mu$, i.e., $y_{ij}\sim N(\mu_{i},\sigma^{2}\varphi(\mu_{i}))$,
with $\varphi$ a positive function that is twice differentiable,
and $\sigma^{2}$ known, the Fisher information matrix $I_{3\times3}(\boldsymbol{\beta})$
becomes equal to (Appendix C):
\begin{equation}
\sum_{i=1}^{3}n_{i}h(\mu_{i})\begin{pmatrix}1 & x_{i}^{\beta_{3}} & \beta_{2}x_{i}^{\beta_{3}}\log(x_{i})\\
x_{i}^{\beta_{3}} & x_{i}^{2\beta_{3}} & \beta_{2}x_{i}^{2\beta_{3}}\log(x_{i})\\
\beta_{2}x_{i}^{\beta_{3}}\log(x_{i}) & \beta_{2}x_{i}^{2\beta_{3}}\log(x_{i}) & [\beta_{2}x_{i}^{\beta_{3}}\log(x_{i})]^{2}
\end{pmatrix}\label{eq:WI_3}
\end{equation}
with $h(\mu)=0.5[\varphi'(\mu)/\varphi(\mu)]^{2}+[\sigma^{2}\varphi(\mu)]^{-1}$,
$\varphi'$ the first derivative of $\varphi$, and with the summation
in (\ref{eq:WI_3}) taken element wise. This Fisher information matrix
has strong similarities with the Fisher information matrix for constructing
optimal designs for the Michaelis-Menten curve studied in \cite{key-30}.
If we now take the weight $w_{i}=n_{i}[\mathsf{VAR}(y_{ij})]^{-1}$
and consider $\boldsymbol{X}^{T}\boldsymbol{W}\boldsymbol{X}$, we
obtain matrix (\ref{eq:WI_3}) with $h(\mu_{i})$ equal to $[\sigma^{2}\varphi(\mu_{i})]^{-1}$.
Thus under heteroscedasticity, the usual inverse variance weight $w_{i}=n_{i}[\mathsf{VAR}(y_{ij})]^{-1}$
does not lead to a D-optimal design, but if we choose the weight $w_{i}=n_{i}h(\mu_{i})$,
with $h(\mu_{i})$ as defined in (\ref{eq:WI_3}), $\boldsymbol{X}^{T}\boldsymbol{W}\boldsymbol{X}$
becomes equal to (\ref{eq:WI_3}).

The determinant of $I_{3\times3}(\boldsymbol{\beta})$ in (\ref{eq:WI_3})
becomes equal to (\ref{eq:D}) with $g'(\mu_{i})$ replaced by $h(\mu_{i})$,
making use of Matlab. Thus the solutions $x_{1}$, $x_{2}$, and $x_{3}$
that maximize $|I_{3\times3}(\boldsymbol{\beta})|$ are determined
by our Theorems \ref{th-x1}, \ref{th-x3}, and \ref{th-x2} when
the following three conditions are satisfied $h(\mu)\geq0$, $h'(\mu)\leq0$,
and $\mu h'(\mu)+2h(\mu)\geq0$. If we would assume that the dispersion
parameter is a power function of the mean, i.e., $\varphi(\mu)=\mu^{p}$,
with $p>0$, we obtain that $h(\mu)=0.5p^{2}\mu^{-2}+\sigma^{-2}\mu^{-p}$,
$h'(\mu)=-p[p\mu^{-3}+\sigma^{-2}\mu^{-p-1}]$, and $\mu h'(\mu)+2h(\mu)=[2-p]\sigma^{-2}\mu^{-p}$.
Thus conditions $h(\mu)\geq0$ and $h'(\mu)\leq0$ are always satisfied
when $p>0$, but condition $\mu h'(\mu)+2h(\mu)\geq0$ is only satisfied
when $0<p\leq2$, implying that we would only choose $x_{3}$ equal
to its maximum value when $p\leq2$. When we assume that $x_{1}^{\mathrm{opt}}=0$
and $p\in(0,2]$, the optimal solution $x_{2}^{\mathrm{opt}}$ follows
from (\ref{eq:optimal-x2-2}) with $g'=h$, and should satisfy equation
\[
\beta_{3}\left[\dfrac{(2-p)}{\sigma^{2}\mu_{2}^{p}}+\dfrac{p\beta_{1}}{\sigma^{2}\mu_{2}^{p+1}}+\dfrac{p^{2}\beta_{1}}{\mu_{2}^{3}}\right]\log\left(\dfrac{x_{3}}{x_{2}}\right)=\dfrac{p^{2}}{\mu_{2}^{2}}+\dfrac{2}{\sigma^{2}\mu_{2}^{p}},
\]
with $\mu_{2}=\beta_{1}+\beta_{2}x_{2}^{\beta_{3}}.$

Unfortunately, when $\sigma^{2}$ would be unknown, the D-optimal
design is not determined by determinant $I_{3\times3}(\boldsymbol{\beta})$
anymore, even though condition (\ref{eq:Independence}) is still satisfied.
Indeed, the variance of the score function with respect to $\sigma$
is independent of $\boldsymbol{\beta}$, since $\mathsf{VAR}(\ell_{\sigma}^{\prime})=0.5n/\sigma^{4}$
(Appendix C). The issue is that the covariances between the score
functions with respect to $\boldsymbol{\beta}$ and the score function
with respect to $\sigma$ are no longer equal to zero (Appendix C).

\subsection{Transformations of the Mitscherlich function}

Our results provide D-optimal designs for $\mathbb{E}(y_{ij}|x_{i})=\beta_{1}+\beta_{2}x_{i}^{\beta_{3}}$,
with $\beta_{1}\geq0$, $\beta_{2},\beta_{3}>0$, $x_{i}\geq0$, and
$y_{ij}$ having a distribution in the canonical exponential family
with $\theta_{i}=g(\beta_{1}+\beta_{2}x_{i}^{\beta_{3}})$. If we
wish to study the non-linear function $\mathbb{E}(y_{ij}|x_{i})=\psi(\beta_{1}+\beta_{2}x_{i}^{\beta_{3}})$,
with canonical link function $g$, the log-likelihood function in
(\ref{eq:likelihood}) and determinant $|I_{3\times3}(\boldsymbol{\beta})|$
in (\ref{eq:D}) both change due to the transformation $\psi$. The
determinant becomes
\[
\beta_{2}^{2}\left[M(\boldsymbol{x}|\boldsymbol{\beta})\right]^{2}\prod_{i=1}^{3}\left[n_{i}g'(\psi(\mu_{i}))\left\{ \psi'(\mu_{i})\right\} ^{2}\right],
\]
with $M(\boldsymbol{x}|\boldsymbol{\beta})=(x_{1}x_{2})^{\beta_{3}}\log(x_{2}/x_{1})-(x_{1}x_{3})^{\beta_{3}}\log(x_{3}/x_{1})+(x_{2}x_{3})^{\beta_{3}}\log(x_{3}/x_{2})$,
$\mu_{i}=\beta_{1}+\beta_{2}x_{i}^{\beta_{3}},$ and $\psi'$ the
derivative of $\psi$. If we now define the function $\tilde{g}$
through its derivative $\tilde{g}'(\mu)=g'(\psi(\mu))[\psi'(\mu)]^{2}$,
the second derivative of function $\tilde{g}$ would become $\tilde{g}''(\mu)=g''(\psi(\mu))[\psi'(\mu)]^{3}+2g'(\psi(\mu))\psi''(\mu)\psi'(\mu)$.
Based on the proof of Theorem \ref{th-x1}, $x_{1}^{\mathrm{opt}}$
should be chosen equal to $L$ when $\tilde{g}'(\mu)=g'(\psi(\mu))[\psi'(\mu)]^{2}$
is non-negative and decreasing in $\mu$ (i.e., $g''(\psi(\mu))[\psi'(\mu)]^{3}+2g'(\psi(\mu))\psi''(\mu)\psi'(\mu)\leq0$).
If we also have that $\mu\tilde{g}''(\mu)+2\tilde{g}'(\mu)\geq0$,
$x_{3}^{\mathrm{opt}}$ should be chosen equal to $U$. The optimal
stimuli $x_{2}^{\mathrm{opt}}$ should satisfy equation (\ref{eq:optimal-x2})
with $g'(\mu_{2})$ and $g''(\mu_{2})$ replaced by $\tilde{g}'(\mu_{2})$
and $\tilde{g}''(\mu_{2})$, respectively.

To illustrate these results, let's assume we would like to study the
square root transformation $\psi(\mu)=\sqrt{\mu}$ of the Mitscherlich
function $\mu=\beta_{1}+\beta_{2}x^{\beta_{3}}$ and assume that the
canonical link function is equal to $g(\mu)=\log(\mu)$, with $\mu>0$.
Then the function $\tilde{g}'(\mu)=0.25\mu^{-3/2}$ is a positive
decreasing function and $\tilde{g}''(\mu)=-0.375\mu^{-5/2}$ is negative
for all $\mu>0.$ Furthermore, the condition $\mu\tilde{g}''(\mu)+2\tilde{g}'(\mu)$
is equal to $0.125\mu^{-3/2}$ and positive for all $\mu>0$. Thus
the combination $g(\mu)=\log(\mu)$ and $\psi(\mu)=\sqrt{\mu}$ results
into $x_{1}^{\mathrm{opt}}=L$, $x_{3}^{\mathrm{opt}}=U$, and $x_{2}^{\mathrm{opt}}$
can be obtained from
\[
\begin{split}-0.75\beta_{2}\beta_{3}\left[(Lx_{2})^{\beta_{3}}\log(\tfrac{x_{2}}{L})-(LU)^{\beta_{3}}\log(\tfrac{U}{L})+(x_{2}U)^{\beta_{3}}\log(\tfrac{U}{x_{2}})\right]\\
+\mu_{2}\left[\beta_{3}L^{\beta_{3}}\log(\tfrac{x_{2}}{L})+\beta_{3}U^{\beta_{3}}\log(\tfrac{U}{x_{2}})+L^{\beta_{3}}-U^{\beta_{3}}\right]=0.
\end{split}
\]
On the other hand, when we would like to study the exponential transformation
$\psi(\mu)=\exp\{\mu\}$ and the canonical link function is still
$g(\mu)=\log(\mu)$, with $\mu>0$, we do not satisfy the criteria.
The function $\tilde{g}'(\mu)=\exp\{\mu\}$ is still positive, but
it is clearly not a decreasing function, since the derivative $\tilde{g}''(\mu)=\exp\{\mu\}$
is always positive. Thus D-optimal designs for the combination $\psi(\mu)=\exp\{\mu\}$
and $g(\mu)=\log(\mu)$ may be different from what Theorems \ref{th-x1},
\ref{th-x3}, and \ref{th-x2} seem to indicate. Thus our theorems
only apply to certain combinations.

\section{Summary and Discussion}

In this paper, we determined D-optimal designs for responses having
their distribution in the exponential family and their mean equal
to the three-parameter Mitscherlich non-linear function, Transformations
of the Mitscherlich function are possible too, but only under certain
conditions. The D-optimal criterion is independent of estimation of
the dispersion parameter if the precision of the estimation of the
dispersion parameter is independent of the parameters of the Mitscherlich
function, a condition that holds for all well-known distribution functions.
It would be interesting to know if there exists an example within
our formulation of the exponential family for which this condition
does not hold.

It was demonstrated that the canonical link function plays an important
role in selecting the optimal values of the three stimuli. For most
distribution functions we should choose the first stimulus as small
as possible and the third stimulus as large as possible, but for the
inverse Gaussian distribution the third stimulus can be substantially
smaller than the maximum allowable stimulus. The middle stimulus depends
on the parameters, the optimal first and third stimulus, and the canonical
link function. For the Binomial distribution, the conditions on the
canonical (logit) link function in our theory may also not always
be satisfied. Hence, the canonical link function has a strong effect
on how to choose the stimuli (as we illustrated).

We showed that our results are an extension of earlier results \cite{key-1},
and that their approach of linearization of the Mitscherlich function
can be extended easily by including weights. Our D-optimality criterion
is identical to the D-optimality criterion of a weighted linear regression
problem, where the weights are the traditional inverse variances of
the response at the selected stimuli. However, when the residual variance
would be heterogeneous, linearization of the Mitscherlich function
does not lead to a D-optimal design anymore.

We believe that the Mitscherlich non-linear function has not received
enough attention, while we believe it is a very useful stimulus-response
function for validation studies of measurement systems. The Mitscherlich
function is an extension of the two-parameter linear or log-linear
regression function and therefore useful to investigate linearity
of the system. Our results may help formulate an optimal design for
maximizing the precision of the estimators of the parameters of the
Mitscherlich function and then evaluate linearity of the measurement
system.

More research in the future can be done to investigate how changing the form of the model would affect the D-optimum points. For example, for testing linearity
we may require alternative optimal designs, since two different models
are being compared that would not have the same D-optimal design.
Secondly, it would be of interest to determine the optimal settings
in case more than three stimuli are being selected, for instance to
test the goodness-of-fit of the Mitscherlich function. Thirdly, more
work is needed to understand the optimal designs for transformations
that may not satisfy our conditions. Finally, we believe that our
work may be extended to other non-linear functions that have similar
characteristics as the Mitscherlich function.

\section*{Acknowledgments}

This work is part of the research program Rapid Micro Statistics with
project number 15990, which is (partly) financed by the Netherlands
Organization for Scientific Research (NWO). The authors gratefully
acknowledge the support of the user committee and the funding organization.

\section*{Appendix A}

Here we will determine the Fisher information matrix for the parameters
$\boldsymbol{\beta}$, using the second derivatives of the log likelihood
function. An explicit formula for determinant $|I_{3\times3}(\boldsymbol{\beta}$)|
in (\ref{eq:D}) can then be determined using for instance Matlab.

The second derivative $\partial^{2}\ell(\boldsymbol{\beta},\phi{\rm |}\boldsymbol{y})/(\partial\beta_{k})^{2}$
of the log likelihood function in (\ref{eq:likelihood}) with respect
to $\beta_{k}$ is given by
\[
-\tfrac{1}{a(\phi)}\sum\limits _{i=1}^{m}\left[n_{i}g'(\mu_{i})\left(\tfrac{\partial\mu_{i}}{\partial\beta_{k}}\right)^{2}-\left\{ g''(\mu_{i})\left(\tfrac{\partial\mu_{i}}{\partial\beta_{k}}\right)^{2}+g'(\mu_{i})\tfrac{\partial^{2}\mu_{i}}{(\partial\beta_{k})^{2}}\right\} \left(y_{i.}-n_{i}\mu_{i}\right)\right].
\]
Using Theorem 1.1 on page 406 of \cite{key-32}, we obtain that the
variance of the score function $\ell_{\beta_{k}}^{\prime}$ is equal
to $\mathrm{\mathsf{VAR}}(\ell_{\beta_{k}}^{\prime})=-\mathbb{E}[\partial^{2}\ell(\boldsymbol{\beta},\phi{\rm |}\boldsymbol{y})/(\partial\beta_{k})^{2}]$.
This leads to the first variance in (\ref{eq:Fisher-Information}),
since $\mathbb{E}[y_{i.}-n_{i}\mu_{i}]=0$.

The second derivative $\partial^{2}\ell(\boldsymbol{\beta},\phi{\rm |}\boldsymbol{y})/(\partial\phi)^{2}$
of the log likelihood function in (\ref{eq:likelihood}) with respect
to $\phi$ is given by
\[
-\sum\limits _{i=1}^{m}\sum\limits _{j=1}^{n_{i}}\left[\tfrac{a''(\phi)a(\phi)-2[a'(\phi)]^{2}}{a^{3}(\phi)}\left(y_{ij}g(\mu_{i})-b(g(\mu_{i}))\right)-c''(y_{ij},\phi)\right],
\]
with $c''(\cdot,\cdot)$ the second derivative with respect to the
second argument. Using again Theorem 1.1 on page 406 of \cite{key-32},
the variance of the score function $\ell_{\phi}^{\prime}$ is equal
to $\mathsf{VAR}(\ell_{\phi}^{\prime})=-\mathbb{E}[\partial^{2}\ell(\boldsymbol{\beta},\phi{\rm |}\boldsymbol{y})/(\partial\phi)^{2}]$.
Since $\mathrm{\mathbb{E}}c'(y_{ij},\phi)=a'(\phi)\mathbb{E}[y_{ij}g(\mu_{i})-b(g(\mu_{i}))]/a^{2}(\phi)$,
with $c'(y_{ij},\phi)=\partial c(y_{ij},\phi)/\partial\phi$, the
variance $\mathsf{VAR}(\ell_{\phi}^{\prime})$ becomes equal to the
second equation in (\ref{eq:Fisher-Information}).

The second derivative $\partial^{2}\ell(\boldsymbol{\beta},\phi{\rm |}\boldsymbol{y})/(\partial\beta_{r}\partial\beta_{s})$
of the log likelihood function in (\ref{eq:likelihood}) with respect
to $\beta_{r}$ and $\beta_{s}$, with $r\neq s$, is given by
\[
\tfrac{1}{a(\phi)}\sum_{i=1}^{m}\left[-n_{i}g'(\mu_{i})\tfrac{\partial\mu_{i}}{\partial\beta_{r}}\tfrac{\partial\mu_{i}}{\partial\beta_{s}}+(y_{i.}-n_{i}\mu_{i})\left(g''(\mu_{i})\tfrac{\partial\mu_{i}}{\partial\beta_{r}}\tfrac{\partial\mu_{i}}{\partial\beta_{s}}+g'(\mu_{i})\tfrac{\partial^{2}\mu_{i}}{\partial\beta_{r}\partial\beta_{s}}\right)\right].
\]
Using the same Theorem 1.1 on page 406 of \cite{key-32}, the covariance
of the score functions $\ell_{\beta_{r}}^{\prime}$ and $\ell_{\beta_{s}}^{\prime}$
is equal to $\mathsf{COV}(\ell_{\beta_{r}}^{\prime},\ell_{\beta_{s}}^{\prime})=-\mathbb{E}[\partial^{2}\ell(\boldsymbol{\beta},\phi{\rm |}\boldsymbol{y})/(\partial\beta_{r}\partial\beta_{s})]$.
Using again that $\mathbb{E}[y_{i.}-n_{i}\mu_{i}]=0$, the covariance
$\mathsf{COV}(\ell_{\beta_{r}}^{\prime},\ell_{\beta_{s}}^{\prime})$
becomes equal to the third equation in (\ref{eq:Fisher-Information}).

The second derivative $\partial^{2}\ell(\boldsymbol{\beta},\phi{\rm |}\boldsymbol{y})/(\partial\beta_{k}\partial\phi)$
of the log likelihood function in (\ref{eq:likelihood}) with respect
to $\beta_{k}$ and $\phi$ is given by
\[
\tfrac{-a'(\phi)}{a^{2}(\phi)}\sum_{i=1}^{m}\left[\left(y_{i.}-n_{i}\mu_{i}\right)g'(\mu_{i})\tfrac{\partial\mu_{i}}{\partial\beta_{k}}\right].
\]
Using the Theorem 1.1 on page 406 of \cite{key-32}, the covariance
of the score functions $\ell_{\beta_{k}}^{\prime}$ and $\ell_{\phi}^{\prime}$
is equal to $\mathsf{COV}(\ell_{\beta_{k}}^{\prime},\ell_{\phi}^{\prime})=-\mathbb{E}[\partial^{2}\ell(\boldsymbol{\beta},\phi{\rm |}\boldsymbol{y})/(\partial\beta_{k}\partial\phi)]$.
Since $\mathbb{E}y_{i.}=n_{i}\mu_{i}$, the covariance $\mathsf{COV}(\ell_{\beta_{k}}^{\prime},\ell_{\phi}^{\prime})$
becomes equal to zero.

\section*{Appendix B}

Here we will provide the proofs of Theorems \ref{th-x1}, \ref{th-x3},
and \ref{th-x2}.

\subsection*{Proof of Theorem \ref{th-x1}}

We may assume that $n_{1}=n_{2}=n_{3}=1$ without loss of generality
and introduce $z_{i}=x_{i}^{\beta_{3}}$ with $\beta_{3}>0$. The
determinant $|I_{3\times3}(\boldsymbol{\beta})|$ can now be written
as
\[
|I_{3\times3}(\boldsymbol{\beta})|=\beta_{3}^{-2}\beta_{2}^{2}g'(\beta_{1}+\beta_{2}z_{1})g'(\beta_{1}+\beta_{2}z_{2})g'(\beta_{1}+\beta_{2}z_{3})[h(z_{1}|z_{2},z_{3})]^{2},
\]

with function $h(z_{1}|z_{2},z_{3})=z_{1}z_{2}\log(\frac{z_{2}}{z_{1}})-z_{1}z_{3}\log(\frac{z_{3}}{z_{1}})+z_{2}z_{3}\log(\frac{z_{3}}{z_{2}})$.
We will demonstrate that the determinant $|I_{3\times3}(\boldsymbol{\beta})|$
is a decreasing function of $z_{1}\in[0,z_{2})$ for any value of
$z_{2}>0$ and $z_{3}>z_{2}$, which proves that we should choose
$x_{1}$ as small as possible.

Since we assumed that $g''(\mu)\leq0$, the function $g'(\mu)$ is
a non-increasing function in $\mu$. Since $\mu_{1}$ is an increasing
function in $z_{1}$ ($\beta_{2}>0$), we have demonstrated that $g'(\mu_{1})$
is a non-increasing function in $z_{1}$ and thus in $x_{1}$. We
will now demonstrate that $h(z_{1}|z_{2},z_{3})$ is decreasing in
$z_{1}$ by showing that $\partial h(z_{1}|z_{2},z_{3})/\partial z_{1}\leq0$
for all $z_{1}\in[0,z_{2})$.

The derivative of $h(z_{1}|z_{2},z_{3})$ with respect to $z_{1}$
is given by
\[
h'(z_{1}|z_{2},z_{3})=(z_{3}-z_{2})\log(z_{1})+z_{2}(\log(z_{2})-1)-z_{3}(\log(z_{3})-1),
\]
which is an increasing function in $z_{1}$. It is clear that $\lim_{z_{1}\downarrow0}h'(z_{1}|z_{2},z_{3})=-\infty$,
that $h'(z_{1}|z_{2},z_{3})=0$ at $z_{1}=z_{1}^{0}$ with 
\[
\log(z_{1}^{0})=\frac{z_{3}(\log(z_{3})-1)-z_{2}(\log(z_{2})-1)}{z_{3}-z_{2}},
\]
and that $h'(z_{1}|z_{2},z_{3})$ is negative for $z_{1}\in[0,z_{1}^{0})$.
If we can show that $\log(z_{1}^{0})\ge\log(z_{2})$, then we can
conclude that $h(z_{1}|z_{2},z_{3})$ is a decreasing function on
the interval $[0,z_{2})$. Inequality $\log(z_{1}^{0})\ge\log(z_{2})$
is identical to $z_{3}(\log(z_{3})-\log(z_{2}))\ge z_{3}-z_{2}$,
using standard algebra. If we now choose $z_{2}=az_{3}$, with $0<a<1$,
inequality $\log(z_{1}^{0})\ge\log(z_{2})$ results in inequality
$a-\log(a)\ge1.$ Since $a-\log(a)$ is a decreasing function for
$a\in(0,1)$ and $a-\log(a)$ is equal to one for $a=1$, we have
demonstrated that $z_{1}^{0}>z_{2}$ and that $h(z_{1}|z_{2},z_{3})$
is a decreasing function in $z_{1}$.

Furthermore, $h(z_{2}|z_{2},z_{3})$ is equal to zero, which means
that $h(z_{1}|z_{2},z_{3})>0$ for $z_{1}\in[0,z_{2})$ and hence
$[h(z_{1}|z_{2},z_{3})]^{2}$ is a decreasing function in $z_{1}$
on the interval $[0,z_{2})$. This implies that $g'(\beta_{1}+\beta_{2}z_{1})[h(z_{1}|z_{2},z_{3})]^{2}$
is a decreasing function in $z_{1}$ on the interval $[0,z_{2})$
and thus $|I_{3\times3}(\boldsymbol{\beta})|$ is a decreasing function
in $x_{1}$ on the interval $[0,x_{2})$.

\subsection*{Proof of Theorem \ref{th-x3}}

Again we assume that $n_{1}=n_{2}=n_{3}=1$ and introduce $z_{i}=x_{i}^{\beta_{3}}$
with $\beta_{3}>0$. The determinant $|I_{3\times3}(\boldsymbol{\beta})|$
can now be written as
\[
|I_{3\times3}(\boldsymbol{\beta})|=\beta_{3}^{-2}\beta_{2}^{2}g'(\beta_{1}+\beta_{2}z_{1})g'(\beta_{1}+\beta_{2}z_{2})g'(\beta_{1}+\beta_{2}z_{3})[h(z_{3}|z_{1},z_{2})]^{2},
\]

with function $h(z_{3}|z_{1},z_{2})=z_{1}z_{2}\log(\frac{z_{2}}{z_{1}})-z_{1}z_{3}\log(\frac{z_{3}}{z_{1}})+z_{2}z_{3}\log(\frac{z_{3}}{z_{2}})$.
We will demonstrate that determinant $|I_{3\times3}(\boldsymbol{\beta})|$
is an increasing function in $z_{3}\in(z_{2},\infty)$, under the
stated conditions of Theorem \ref{th-x3} for any value of $z_{1}\geq0$
and $z_{2}>z_{1}$, which proves that we should choose $x_{3}$ as
large as possible. To do this, we may just study the product function
$D(z_{3})=g'(\beta_{1}+\beta_{2}z_{3})[h(z_{3}|z_{1},z_{2})]^{2}$,
since all other elements in $|I_{3\times3}(\boldsymbol{\beta})|$
are positive constants with respect to $z_{3}$.

If we denote $h'(z_{3}|z_{1},z_{2})$ as the derivative of $h(z_{3}|z_{1},z_{2})$
with respect to $z_{3}$ and define $C(\mu)=g''(\mu)\mu+2g'(\mu)$,
the derivative of $D(z_{3})$ with respect to $z_{3}$ can be written
as
\begin{equation}
\begin{array}{rcl}
\frac{\partial D(z_{3})}{\partial z_{3}} & = & \beta_{2}g''(\mu_{3})\left[h(z_{3}|z_{1},z_{2})\right]^{2}+2g'(\mu_{3})h(z_{3}|z_{1},z_{2})h'(z_{3}|z_{1},z_{2})\\
 & = & \frac{\left[h(z_{3}|z_{1},z_{2})\right]^{2}}{z_{3}}\left[g''(\mu_{3})\beta_{2}z_{3}+2g'(\mu_{3})\frac{z_{3}h'(z_{3}|z_{1},z_{2})}{h(z_{3}|z_{1},z_{2})}\right]\\
 & = & \frac{\left[h(z_{3}|z_{1},z_{2})\right]^{2}}{z_{3}}\left[g''(\mu_{3})\mu_{3}-\beta_{1}g''(\mu_{3})+2g'(\mu_{3})\frac{z_{3}h'(z_{3}|z_{1},z_{2})}{h(z_{3}|z_{1},z_{2})}\right]\\
 & = & \frac{\left[h(z_{3}|z_{1},z_{2})\right]^{2}}{z_{3}}\left[C(\mu_{3})-\beta_{1}g''(\mu_{3})+2g'(\mu_{3})\left(\frac{z_{3}h'(z_{3}|z_{1},z_{2})}{h(z_{3}|z_{1},z_{2})}-1\right)\right]
\end{array}
\end{equation}

Based on the assumptions of Theorem \ref{th-x3}, we have that $C(\mu_{3})\geq0$,
$-\beta_{1}g''(\mu_{3})\geq0$, and $2g'(\mu_{3})\geq0$. Furthermore,
the term $\left[h(z_{3}|z_{1},z_{2})\right]^{2}/z_{3}$ is non-negative
for all $z_{3}$. If we can now demonstrate that the term $z_{3}h'(z_{3}|z_{1},z_{2})/h(z_{3}|z_{1},z_{2})\geq1$,
we have demonstrated that the derivative $\partial D(z_{3})/\partial z_{3}$
is non-negative, indicating that $D(z_{3})$ is increasing.

If we rewrite $h(z_{3}|z_{1},z_{2})$ into $h(z_{3}|z_{1},z_{2})=A_{1}+A_{2}z_{3}\log(z_{3})+A_{3}z_{3}$,
with $A_{1}=z_{1}z_{2}\log(\frac{z_{2}}{z_{1}})$, $A_{2}=z_{2}-z_{1}$,
and $A_{3}=z_{1}\log(z_{1})-z_{2}\log(z_{2})$, the derivative of
$h(z_{3}|z_{1},z_{2})$ with respect to $z_{3}$ is given by $h'(z_{3}|z_{1},z_{2})=A_{2}+A_{2}\log(z_{3})+A_{3}$.
Since $A_{2}>0$, $h'(z_{3}|z_{1},z_{2})$ is an increasing function
with $\lim_{z_{3}\rightarrow\infty}h'(z_{3}|z_{1},z_{2})=\infty$.
The function $h'(z_{3}|z_{1},z_{2})$ is equal to zero in $z_{3}=z_{3}^{0}$,
with
\[
\log(z_{3}^{0})=-\frac{A_{3}+A_{2}}{A_{2}}=\frac{z_{2}\log(z_{2})-z_{1}\log(z_{1})}{z_{2}-z_{1}}-1.
\]
If we can demonstrate that $\log(z_{3}^{0})\leq\log(z_{2})$, we would
obtain that $h'(z_{3}|z_{1},z_{2})>0$ for $z_{3}\in(z_{2},\infty)$,
and thus $h(z_{3}|z_{1},z_{2})$ is an increasing function. Since
$h(z_{2}|z_{1},z_{2})=0$, $h(z_{3}|z_{1},z_{2})$ would also be positive
on $(z_{2},\infty)$.

If we now choose $z_{1}=az_{2}$, with $a\in[0,1)$, the solution
$\log(z_{3}^{0})$ is equal to $[\log(z_{2})-a\log(z_{2})-a\log(a)-1+a]/[1-a]$.
Then inequality $\log(z_{3}^{0})\leq\log(z_{2})$ results in $a-a\log(a)\leq1$.
The function $a-a\log(a)$ is an increasing function in $a\in[0,1)$,
since its derivative is equal to $-\log(a)$, and it is equal to one
when $a=1$. Thus inequality $\log(z_{3}^{0})\leq\log(z_{2})$ is
guaranteed.

Knowing that $h(z_{3}|z_{1},z_{2})>0$ for $z_{3}\in(z_{2},\infty)$,
we can see that inequality $z_{3}h'(z_{3}|z_{1},z_{2})/h(z_{3}|z_{1},z_{2})\geq1$
is equivalent to the following inequality
\[
\begin{array}{rcl}
z_{3}h'(z_{3}|z_{1},z_{2})/h(z_{3}|z_{1},z_{2})\geq1 & \Longleftrightarrow & A_{2}z_{3}-A_{1}\ge0\\
 & \Longleftrightarrow & z_{3}\geq z_{1}z_{2}\log(z_{2}/z_{1})/[z_{2}-z_{1}]
\end{array}
\]
If we can prove that $z_{1}z_{2}\log(z_{2}/z_{1})/[z_{2}-z_{1}]$
is smaller than or equal to $z_{2}$, we have demonstrated that $z_{3}h'(z_{3}|z_{1},z_{2})/h(z_{3}|z_{1},z_{2})\geq1$
holds for $z_{3}\in(z_{2},\infty)$. If we again take $z_{1}=az_{2}$,
with $a\in[0,1)$, we obtain that $z_{1}z_{2}\log(z_{2}/z_{1})/[z_{2}-z_{1}]=-az_{2}\log(a)/[1-a]$
and this function is smaller or equal to $z_{2}$ when $-a\log(a)/[1-a]\leq1$.
This results again in $a-a\log(a)\leq1$, which we already demonstrated
to be true.

Thus we have finally shown that $\partial D(z_{3})/\partial z_{3}>0$
under the stated conditions of Theorem \ref{th-x3}, making the determinant
$|I_{3\times3}(\boldsymbol{\beta})|$ an increasing function in $z_{3}$
on the interval $(z_{2},\infty)$.

\subsection*{Proof of Theorem \ref{th-x2}}

We start again with the assumption that $n_{1}=n_{2}=n_{3}=1$ and
we introduce $z_{i}=x_{i}^{\beta_{3}}$ with $\beta_{3}>0$. The determinant
$|I_{3\times3}(\boldsymbol{\beta})|$ can now be written as
\[
|I_{3\times3}(\boldsymbol{\beta})|=\beta_{3}^{-2}\beta_{2}^{2}g'(\beta_{1}+\beta_{2}z_{1})g'(\beta_{1}+\beta_{2}z_{2})g'(\beta_{1}+\beta_{2}z_{3})[h(z_{2}|z_{1},z_{3})]^{2},
\]

with function $h(z_{2}|z_{1},z_{3})=z_{1}z_{2}\log(\frac{z_{2}}{z_{1}})-z_{1}z_{3}\log(\frac{z_{3}}{z_{1}})+z_{2}z_{3}\log(\frac{z_{3}}{z_{2}})$.
We will now study $D(z_{2})=g'(\beta_{1}+\beta_{2}z_{2})[h(z_{2}|z_{1},z_{3})]^{2}$
as function of $z_{2}$ in the interval $(z_{1},z_{3})$, since the
remaining part of the determinant is just a constant.

Function $h(z_{2}|z_{1},z_{3})$ is rewritten into $h(z_{2}|z_{1},z_{3})=-A_{1}z_{2}\log(z_{2})+A_{2}z_{2}-A_{3}$,
with $A_{1}=z_{3}-z_{1}$, $A_{2}=z_{3}\log(z_{3})-z_{1}\log(z_{1})$,
and $A_{3}=z_{1}z_{3}\log(z_{3}/z_{1})$. The terms $A_{1}$ and $A_{3}$
are both positive when $z_{3}>z_{1}\geq0$, but the term $A_{2}$
is only positive when we may assume that $z_{3}\geq1$. Note that
$h(z_{1}|z_{1},z_{3})=h(z_{3}|z_{1},z_{3})=0$ and the derivative
$h'(z_{2}|z_{1},z_{3})=\partial h(z_{2}|z_{1},z_{3})/\partial z_{2}$
is equal to $A_{2}-A_{1}-A_{1}\log(z_{2})$, which is a decreasing
function in $z_{2}$. The solution $z_{2}^{0}$ of $h'(z_{2}|z_{1},z_{3})=0$
is unique and satisfies $\log(z_{2}^{0})=[A_{2}-A_{1}]/A_{1}$. If
we can demonstrate that $\log(z_{1})<\log(z_{2}^{0})<\log(z_{3})$,
we would know that $h(z_{2}|z_{1},z_{3})$ is increasing on interval
$z_{2}\in(z_{1},z_{2}^{0})$ and decreasing on interval $z_{2}\in(z_{2}^{0},z_{3})$
and thus always positive on $z_{2}\in(z_{1},z_{3})$.

Inequality $\log(z_{1})<\log(z_{2}^{0})$ is equivalent to inequality
$z_{3}-z_{1}<z_{3}[\log(z_{3})-\log(z_{1})]$. If we choose $z_{1}=az_{3}$,
with $a\in[0,1)$, the inequality becomes $1-a+\log(a)<0$, which
holds for $a\in[0,1)$, since $1-a+\log(a)$ is an increasing function
on interval $a\in[0,1)$ with $\lim_{a\rightarrow1}[1-a+\log(a)]=0$.
Inequality $\log(z_{2}^{0})<\log(z_{3})$ is equivalent to inequality
$z_{1}[\log(z_{3})-\log(z_{1})]<(z_{3}-z_{1})$. If we choose $z_{1}=az_{3}$,
with $a\in[0,1)$, the inequality becomes $1-a+a\log(a)>0$, which
holds for $a\in[0,1)$, since $1-a+a\log(a)$ is a decreasing function
on interval $a\in[0,1)$ with $\lim_{a\rightarrow1}[1-a+a\log(a)]=0$.
Thus this proves $\log(z_{1})<\log(z_{2}^{0})<\log(z_{3})$.

Maximizing $D(z_{2})$ in $z_{2}$, can be done by setting the derivative
equal to zero. Since $h(z_{2}|z_{1},z_{3})$ is positive on $z_{2}\in(z_{1},z_{3})$,
setting the derivative $\partial D(z_{2})/\partial z_{2}$ equal to
zero leads to the following equality
\begin{equation}
\beta_{2}g''(\mu_{2})h(z_{2}|z_{1},z_{3})+2g'(\mu_{2})h'(z_{2}|z_{1},z_{3})=0,\label{eq:opt-z2}
\end{equation}
which is identical to equation (\ref{eq:optimal-x2}). We now need
to demonstrate that equation (\ref{eq:opt-z2}) has at least one solution
for a value of $z_{2}\in(z_{1},z_{3})$. Rewriting equation (\ref{eq:opt-z2}),
leads to $\beta_{2}g''(\mu_{2})/g'(\mu_{2})=-2h'(z_{2}|z_{1},z_{3})/h(z_{2}|z_{1},z_{3})$.
The left-hand side is smaller or equal to zero for any $z_{2}$, while
the right-hand side is negative for $z_{2}\in(z_{1},z_{2}^{0})$,
zero at $z_{2}=z_{2}^{0}$, and positive for $z_{2}\in(z_{2}^{0},z_{3})$,
using the results on $h'(z_{2}|z_{1},z_{3})$ and $h(z_{2}|z_{1},z_{3})$
above. Since we also have that $\lim_{z_{2}\downarrow z_{1}}-2h'(z_{2}|z_{1},z_{3})/h(z_{2}|z_{1},z_{3})=-\infty$,
we now know that a solution must occur for $z_{2}\in(z_{1},z_{2}^{0}]$,
thereby satisfying the constraint in Theorem \ref{th-x2}.

\section*{Appendix C}

Here we will assume that $y_{ij}\sim N(\mu_{i},\sigma^{2}\varphi(\mu_{i}))$
is normally distributed with a mean $\mu_{i}$ equal to the Mitscherlich
function $\mathbb{E}(y_{ij}|x_{i})\equiv\mu_{i}=\beta_{1}+\beta_{2}x_{i}^{\beta_{3}}$,
and with $\beta_{i}>0$. We will provide the elements of the Fisher
information matrix for estimation of $\boldsymbol{\theta}^{T}=(\beta_{1},\beta_{2},\beta_{3},\sigma^{2})$.
The log likelihood function is given by
\[
\ell(\boldsymbol{\theta}|\boldsymbol{y})=-\tfrac{1}{2}\sum_{i=1}^{m}\sum_{j=1}^{n_{i}}\left[\log(2\pi)+\log(\sigma^{2})+\log(\varphi(\mu_{i}))+(y_{ij}-\mu_{i})^{2}/(\sigma^{2}\varphi(\mu_{i}))\right],
\]
with $\boldsymbol{y}=(\boldsymbol{y}_{1},\boldsymbol{y}_{2},...,\boldsymbol{y}_{m})^{T}$
and $\boldsymbol{y}_{i}=(y_{i1},y_{i2},...,y_{in_{i}})^{T}.$ The
four score functions $\ell_{\beta_{k}}^{\prime}=\partial\ell(\boldsymbol{\theta}|\boldsymbol{y})/\partial\beta_{k}$,
$k\in\{1,2,3\}$, and $\ell_{\sigma^{2}}^{\prime}=\partial\ell(\boldsymbol{\theta}|\boldsymbol{y})/\partial(\sigma^{2})$
are now given by
\[
\begin{array}{rcl}
\ell_{\beta_{k}}^{\prime} & = & -\frac{1}{2}\sum\limits _{i=1}^{m}\sum\limits _{j=1}^{n_{i}}\left[\tfrac{\varphi'(\mu_{i})}{\varphi(\mu_{i})}-2\tfrac{y_{ij}-\mu_{i}}{\sigma^{2}\varphi(\mu_{i})}-\tfrac{\varphi'(\mu_{i})(y_{ij}-\mu_{i})^{2}}{[\sigma\varphi(\mu_{i})]^{2}}\right]\frac{\partial\mu_{i}}{\partial\beta_{k}},\\
\ell_{\sigma^{2}}^{\prime} & = & -\frac{1}{2}\sum\limits _{i=1}^{m}\sum\limits _{j=1}^{n_{i}}\left[\tfrac{1}{\sigma^{2}}-\tfrac{(y_{ij}-\mu_{i})^{2}}{\sigma^{4}\varphi(\mu_{i})}\right].
\end{array}
\]
After tedious algebraic calculations using the score functions directly,
the variances and covariances of the score functions can be calculated.
They are given by 
\[
\begin{array}{rcl}
\mathsf{VAR}(\ell_{\beta_{k}}^{\prime}) & = & \sum\limits _{i=1}^{m}n_{i}\left[\tfrac{1}{2}\left(\tfrac{\varphi'(\mu_{i})}{\varphi(\mu_{i})}\right)^{2}+\tfrac{1}{\sigma^{2}\varphi(\mu_{i})}\right]\left(\frac{\partial\mu_{i}}{\partial\beta_{k}}\right)^{2},\\
\mathsf{VAR}(\ell_{\sigma^{2}}^{\prime}) & = & \frac{1}{2\sigma^{4}}\sum\limits _{i=1}^{m}n_{i},\\
\mathsf{COV}(\ell_{\beta_{r}}^{\prime},\ell_{\beta_{s}}^{\prime}) & = & \sum\limits _{i=1}^{m}n_{i}\left[\tfrac{1}{2}\left(\tfrac{\varphi'(\mu_{i})}{\varphi(\mu_{i})}\right)^{2}+\tfrac{1}{\sigma^{2}\varphi(\mu_{i})}\right]\frac{\partial\mu_{i}}{\partial\beta_{r}}\frac{\partial\mu_{i}}{\partial\beta_{s}},\\
\mathsf{COV}(\ell_{\beta_{k}}^{\prime},\ell_{\sigma^{2}}^{\prime}) & = & \frac{1}{2\sigma^{2}}\sum\limits _{i=1}^{m}n_{i}\left[\tfrac{\varphi'(\mu_{i})}{\varphi(\mu_{i})}\right]\frac{\partial\mu_{i}}{\partial\beta_{k}}.
\end{array}
\]
Since we have $\partial\mu_{i}/\partial\beta_{1}=1$, $\partial\mu_{i}/\partial\beta_{2}=x_{i}^{\beta_{3}}$,
and $\partial\mu_{i}/\partial\beta_{3}=\beta_{2}x_{i}^{\beta_{3}}\log(x_{i})$,
we obtain that the matrix $I_{3\times3}(\boldsymbol{\beta})$ is given
by (\ref{eq:WI_3}).

\end{document}